\documentclass[11pt, twoside]{amsart}
\usepackage[cp1251]{inputenc}
\usepackage{color}
\usepackage{amsthm,amssymb,latexsym}
\usepackage[centertags]{amsmath}
\usepackage[T2A]{fontenc}
\usepackage{graphs} 
\usepackage{MnSymbol}
\usepackage{bbm}
\usepackage{centernot}
\usepackage{mathrsfs}
\usepackage{amsthm}
\usepackage{amsfonts}
\usepackage{textcomp}
\usepackage{newlfont}
\usepackage{enumitem}

\newif\ifPDF
\ifx\pdfoutput\undefined\PDFfalse
\else \ifnum \pdfoutput > 0 \PDFtrue
        \else \PDFfalse
        \fi
\fi

\ifPDF
  \usepackage[pdftex]{xcolor, graphicx}
  \usepackage[colorlinks=true,linkcolor=blue, citecolor=blue]{hyperref}%

\else
  \usepackage{color}
  \usepackage[dvips]{graphicx}
  \usepackage[dvips]{hyperref}
\fi

\usepackage{tkz-graph}
\tikzset{EdgeStyle/.style = {->}}
\tikzset{LabelStyle/.style= {fill=yellow}}
\usetikzlibrary{shapes,snakes,calendar,matrix,backgrounds,folding}

\usepackage{bbm}

\usepackage[top=1in, bottom=1.25in, left=1.25in, right=1.25in]
{geometry}

\theoremstyle{definition}
\newtheorem{definition}{Definition}[section]
\theoremstyle{remark}
\newtheorem{example}[definition]{Example}

\newtheorem{remark}[definition]{Remark}

\theoremstyle{plain}
\newtheorem{thm}[definition]{Theorem}
\newtheorem{theorem}[definition]{Theorem}
\newtheorem{prop}[definition]{Proposition}
\newtheorem{lemma}[definition]{Lemma}

\newtheorem{corol}[definition]{Corollary}

\def\ol{\overline}
\def\ov{\overline}
\def\tl{\widetilde}
\def\wh{\widehat}

\newcommand{\N}{\mathbb N}
\newcommand{\Z}{\mathbb Z}

\newcommand{\be}{\begin{equation}}
\newcommand{\ee}{\end{equation}}
\newcommand{\ba}{\begin{aligned}}
\newcommand{\ea}{\end{aligned}}

\newcommand{\R}{{\mathbb R}}

\numberwithin{equation}{section}
\newcommand{\ignore}[1]{}

\begin{document}

\title{Isomoprhism of generalized Bratteli diagrams}

\author{Olena Karpel}
\address{AGH University of Krakow, Faculty of Applied Mathematics, al. Adama Mickiewicza~30, 30-059 Krak\'ow, Poland \&
B. Verkin Institute for Low Temperature Physics and Engineering,
47~Nauky Ave., Kharkiv, 61103, Ukraine}

\email{okarpel@agh.edu.pl}



\date{}

\begin{abstract}

We study the notion of isomorphism for generalized Bratteli diagrams and investigate properties preserved under isomorphism. We show that every generalized Bratteli diagram is isomorphic to an irreducible generalized Bratteli diagram. We introduce the notion of a completely irreducible generalized Bratteli diagram, namely, a diagram that is isomorphic only to irreducible generalized Bratteli diagrams. We establish connections between this notion and the topological properties of the path space of a generalized Bratteli diagram and its tail equivalence relation. We examine in detail several classes of generalized Bratteli diagrams that illustrate these results.


\end{abstract}

\maketitle
\tableofcontents

\section{Introduction}\label{intro}

Bratteli diagrams are infinite graded graphs introduced by O. Bratteli in \cite{Bratteli1972} for the study of approximately finite-dimensional $C^*$-algebras. Dynamical systems were associated with Bratteli diagrams by A. Vershik in \cite{Vershik_1981, Vershik_1982}. It turned out that Bratteli diagrams and the dynamical systems arising on their path spaces provide a powerful framework for modeling a broad class of equivalence relations and dynamical systems in measurable, Cantor and Borel dynamics. 
One of the fundamental results states that any minimal homeomorphism of a Cantor space can be modelled as a Vershik map acting on the path space of an associated simple  Bratteli diagram \cite{HermanPutnamSkau1992}. 
The combinatorial structure of Bratteli diagrams provides tools for studying invariant measures and orbit structure for the corresponding dynamical systems and equivalence relations. There is vast literature on Bratteli diagrams (see, e.g., the books and surveys \cite{Durand2010, BezuglyiKarpel2016, DownarowiczKarpel2018, Putnam2018, BezuglyiKarpel_2020, DurandPerrin2022} and references therein). 

This paper focuses on the study of generalized Bratteli diagrams, a relatively new notion introduced in \cite{BezuglyiDooleyKwiatkowski2006}. 
In that work, it was shown that any aperiodic Borel automorphism of a standard Borel space 
can be modelled as a Vershik map acting on the path space of an associated generalized Bratteli diagram. The orbits of the constructed Vershik map coincide with the orbits of the so-called tail equivalence relation associated with the diagram.
While introducing the Vershik map requires introducing partial order on the set of edges of a Bratteli diagram, the tail equivalence relation does not require the order and can be introduced for every (standard or generalized) Bratteli diagram.
Moreover, with every generalized Bratteli diagram one can associate a Borel automorphism whose orbits coincide with those of the associated tail equivalence relation \cite{Kechris2024}. 
It also follows from the results of \cite{BezuglyiDooleyKwiatkowski2006} that every aperiodic hyperfinite countable Borel equivalence relation can be represented as a tail equivalence relation on the path space of a generalized Bratteli diagram.
The structure of a Bratteli diagram provides tools for describing invariants for Borel isomorphism of such equivalence relations, and for orbit equivalence of the corresponding Borel automorphisms. One of such invariants is the number of  ergodic invariant probability measures for the tail equivalence relation
\cite{Dougherty_Jackson_Kechris1994, Kechris2024, BezuglyiKarpelKwiatkowski2019, BezuglyiKarpelKwiatkowskiWata2024}.
In this paper, we entirely focus on the tail equivalence relation, and do not consider the corresponding Vershik map.

The systematic study of generalized Bratteli diagrams has started only recently (see, e.g.,  \cite{BezuglyiJorgensenSanadhya2024, BezuglyiJorgensenKarpelSanadhya2025, BezuglyiKarpelKwiatkowskiWata2024,BezuglyiJorgensenKarpelKwiatkowski2025, BezuglyiDudkoKarpel2024, BezuglyiJorgensenKarpelSanadhyaRaszeja2026}). 
Though Bratteli diagrams provide a convenient framework for studying dynamical systems in Cantor and Borel dynamics, there is no unique way to associate a diagram to a transformation. 
In general, it remains an open problem to determine whether given two Bratteli diagrams model isomorphic dynamical systems (see, e.g., \cite{DownarowiczKarpel_2019}). 
In this paper, we consider the natural notion of isomorphism for Bratteli diagrams, which can be found, for instance, in  \cite[Section 2]{HermanPutnamSkau1992} or \cite[Subsection 6.3.1]{Durand2010} for standard Bratteli diagrams, and in \cite{BezuglyiJorgensenKarpelSanadhya2025} for generalized ones. 
The diagrams isomorphic with respect to this natural isomorphism have Borel isomorphic tail equivalence relations. 

We study the invariants of isomorphism of generalized Bratteli diagrams.
In particular, we focus on the notion of irreducibility. For a Bratteli diagram, the set of vertices is partitioned into subsets $V = \bigsqcup_{n = 0}^{\infty} V_n$ called levels. The sets of vertices of each level of a generalized Bratteli diagram is countably infinite, each set $V_n$ can be identified with the same countably infinite set, e.g., the set of natural numbers or the set integers. 
The diagram is called irreducible if 
for any vertices $i, j \in V_0$ and any $n \in \N_0$ there exist 
$m 
> n$ such that there is a finite path between $i \in V_n$ and $j \in V_m$.
In general, irreducibility is not preserved under isomorphism, but we show that there are generalized Bratteli diagrams
which are isomorphic only to irreducible generalized Bratteli diagrams. We call such diagrams completely irreducible.
Complete irreducibility implies minimality of the tail equivalence relation, thus, completely irreducible generalized Bratteli diagrams can be viewed as analogues of simple standard Bratteli diagrams.
We give necessary conditions and sufficient conditions  for a generalized Bratteli diagram to be completely irreducible. 
We prove that every  generalized Bratteli diagram is isomorphic to an irreducible generalized Bratteli diagram. 
We show that irreducibility and connectedness as an undirected graph are independent properties for generalized Bratteli diagrams. We prove that topological transitivity of the tail equivalence relation implies connectedness of the Bratteli diagram, though the converse statement is not true.
The results are illustrated by numerous examples. We note that a similar notion of irreducibility can also be introduced for finite rank standard Bratteli diagrams. The notion of irreducibility is also known for Markov chains, although the tail equivalence relation is not considered there (see e.g. \cite{Seneta2006}).

The outline of the paper is as follows. In Section~\ref{sec:prelim}, we provide the notation and basic definitions that we  use throughout the paper. In Section \ref{sec:examples}, we study in detail several classes of generalized Bratteli diagrams which demonstrate various behaviours in terms of irreducibility and topological properties of the corresponding tail equivalence relation, such as minimality and topological transitivity. These diagrams also
serve later as examples for the results from Section \ref{sec:mainres}. Section \ref{sec:mainres} contains the main results of the paper. We introduce the notions of completely irreducible and relatively irreducible generalized Bratteli diagrams. We show that every generalized Bratteli diagram is either completely or relatively irreducible. We provide necessary conditions and sufficient conditions for a Bratteli diagram to be completely irreducible. 
We study the relations between connectedness, irreducibility, and topological properties of the tail equivalence relation, and of the path space of a generalized Bratteli diagram.
Our main results are contained in Theorems \ref{thm:classII}, \ref{Thm:isomredus}, \ref{thm:Class1_minimalR}.





\section{Preliminaries}\label{sec:prelim}

This section contains the main definitions concerning generalized Bratteli diagrams.  
In particular, we discuss the notions of isomorphism and irreducibility for generalized Bratteli diagrams. For more details on generalized Bratteli diagrams see, e.g., \cite{BezuglyiJorgensenKarpelSanadhya2025} or \cite{BezuglyiKarpelKwiatkowskiWata2024}.

We use the standard notation $\mathbb N, \Z, \R,$  
$\mathbb{N}_0 = \N \cup \{0\}$ for the sets of numbers, 
$|\cdot |$ denotes the cardinality of a set.


\begin{definition}\label{Def:generalized_BD} A 
\textit{(standard or generalized) Bratteli diagram} is a graded graph 
$B = (V, E)$ such that the 
vertex set $V$ and the edge set $E$ can be partitioned $V = \bigsqcup_{n=0}^\infty  V_n$ and $E = 
\bigsqcup_{n=0}^\infty  E_n$ so that the following 
properties hold: 

\begin{enumerate}[label=\upshape(\roman*), leftmargin=*, widest=iii]

\item for every $n \in \mathbb{N}_0$, the number of vertices at each level 
$V_n$ is countable, it is finite for standard Bratteli diagrams and countably infinite for generalized Bratteli diagrams. Similarly, the set $E_n$
of all edges between $V_n$ and $V_{n+1}$ is countable and it is finite for standard Bratteli diagrams and countably infinite for generalized Bratteli diagrams,

\item for every edge $e\in E$, we define the 
range and 
source maps $r$ and $s$ such that $r(E_n) = V_{n+1}$ and 
$s(E_n) = V_{n}$ for $n \in \N_0$,

\item for every vertex $v \in V \setminus V_0$, we 
have $|r^{-1}(v)| < \infty$ (the condition holds automatically for standard Bratteli diagrams).
\end{enumerate}
\end{definition}

\begin{remark}
For all $n$, the edges of $E_n$ always have sources in $V_{n}$ and ranges in $V_{n+1}$, hence in the pictures all the edges should be directed downwards. For the better readability of figures, we draw all Bratteli diagrams as unordered graphs.
\end{remark}

\begin{remark}
    For standard Bratteli diagrams, there is a convention which comes from the paper of O. Bratteli \cite{Bratteli1972} to assume that $V_0 = \{v_0\}$ is a single point. Since this paper mostly focuses on the case of generalized Bratteli diagrams, for the sake of the uniformity of notation, we drop this assumption and allow $V_0$ to be arbitrary finite set in the case of standard Bratteli diagrams. 
\end{remark}

Let $B$ be a (standard or generalized) Bratteli diagram.
The set $V_n$ is called the \textit{$n$th level} of the diagram $B$. For generalized Bratteli diagrams, we usually identify each $V_n$ either with $\N$ (or $\N_0$) or with $\Z$. If the vertices of each level of a generalized Bratteli diagram are enumerated by $\N$ or $\N_0$, we call such diagram \textit{one-sided}, if the the vertices are enumerated by $\Z$, the diagram is called \textit{two-sided}.

To define the \textit{path space} of a generalized Bratteli diagram
$B$, consider 
a (finite or infinite) sequence of edges
 $ x = (e_i: e_i\in E_i)$ (it is called a \textit{path}) such 
that $s(e_i)=r(e_{i-1})$. Denote the set of all infinite
paths $x$ starting at a vertex in $V_0$ by $X_B$ and call 
it the \textit{path space} of the diagram $B$. For a finite path
$\ol e = (e_0, ... , e_n)$, let 
$s(\ol e) = s(e_0)$ and $r(\ol e) = r(e_n)$. The set
$$
    [\ol e] := \{x = (x_i) \in X_B : x_0 = e_0, ..., x_n = e_n\}, 
$$ 
is called the \textit{cylinder set} associated with $\ol e$.

The \textit{topology} on the path space $X_B$ is generated by
cylinder sets that are clopen in this topology.
This topology coincides with the topology defined by the 
following metric on $X_B$: for $x = (x_i), \, y = (y_i)$, set 
$$
\mathrm{dist}(x, y) = \frac{1}{2^N},\ \ \ N = \min\{i \in \N_0 : 
x_i \neq y_i\}.
$$
The path space $X_B$ is a zero-dimensional Polish space and, therefore, a standard Borel space. In general, $X_B$ is not locally compact. 

A matrix $A = (a_{ij})$ is called \textit{infinite} 
(or countably infinite) if its rows and columns
are indexed by the same countably infinite set.
For vertices $v \in V_m$ and $w \in V_{n}$, 
denote 
by $E(v, w)$ the set of all finite paths between $v$ and $w$. 
Set $f^{(n)}_{vw} = |E(v, w)|$ for every $w \in V_n$ and $v \in 
V_{n+1}$. In other words, $f^{(n)}_{vw}$ is a number of edges between the vertices $w \in V_n$ and $v \in V_{n+1}$. In such a way,  we associate with the generalized 
Bratteli diagram $B = (V,E)$ 
a sequence of non-negative countably infinite matrices $(F_n)$, 
$n \in \N_0$, 
called the \textit{incidence  matrices}: 
\begin{equation}\label{Notation:f^n}
    F_n = (f^{(n)}_{vw} : v \in V_{n+1}, w\in V_n),\ \   
    f^{(n)}_{vw}  \in \N_0.
\end{equation}

We will write $B = B(F_n)$. The structure of a (standard or generalized) generalized Bratteli diagram $B=(V, E)$ is 
completely determined by the sequence of its incidence matrices 
$(F_n), 
\, n \in \N_0$.  For each $n 
\in \N_0$, the matrix $F_n$ has at most finitely many non-zero 
entries in each row, and none of its rows or columns are entirely 
zero. For a generalized Bratteli diagram $B$, a column of $F_n$ may have a finite or infinite number of 
non-zero entries. 
If  $F_n = F$ for every $n \in 
\N_0$, then the diagram $B$ is called \textit{stationary.} We 
will write $B = B(F)$ in this case.

\begin{definition}\label{Def:Tail_equiv_relation}
Two paths $x= (x_i)$ and $y=(y_i)$ in $X_B$ are called 
\textit{tail equivalent} if there exists an $n \in \mathbb{N}_0$ 
such that $x_i = y_i$ for all $i \geq n$. This notion defines a \textit{countable Borel equivalence relation} $\mathcal R$ in the path space $X_B$, which is called the \textit{tail equivalence relation}.
\end{definition}

Denote by $\mathcal{O}(x)$ the orbit of a point $x \in X_B$ under the tail equivalence relation:
$$
\mathcal{O}(x) = \{y \in X_B : (x,y) \in \mathcal{R}\}.
$$
The tail equivalence relation is called \textit{minimal} if the orbit of every point is dense, and it is called \textit{topologically transitive} if there exists a point with a dense orbit (see e.g. \cite{KatokHasselblatt1995} or \cite{Bruin2022}). 

The following notion of isomorphism can be found, for instance, in \cite[Section 2]{HermanPutnamSkau1992} and \cite[Subsection 6.3.1]{Durand2010} for standard Bratteli diagrams, and in \cite[Definition 2.8]{BezuglyiJorgensenKarpelSanadhya2025} for generalized Bratteli diagrams.

\begin{definition} \label{def_isom BD}
Two (standard or generalized) Bratteli diagrams $B = (V, E)$ and $B' = 
(V', E')$ 
are called \textit{isomorphic} if there exist two sequence of
bijections $(g_n : V_n \rightarrow V_n')_{n \in \N_0}$ and 
$(h_n : E_n \rightarrow E_n')_{n \in \N_0}$ such that 
for every $n \in \N_0$, we have $g_n(V_n) = V_n'$ and $h_n(E_n) = 
E_n'$, and $s' \circ h_n = g_n \circ s$, $r' \circ h_n = 
g_n \circ r$,  where $s'$ and $r'$ are source and range maps in $B'$.
\end{definition}

In most of the studies concerning standard and generalized Bratteli diagrams, one is interested in the structure of the diagram, which is completely determined by the sequence of incidence matrices. 
Thus, in this paper, we focus on the following definition rather than on the definition of isomorphism:


\begin{definition}
Let $B = (V,E)$ and $B' = (V',E')$ be two (standard or generalized) Bratteli diagrams with the corresponding sequences of incidence matrices $(F_n)$ and $(F_n')$. Assume that for every $n \in \mathbb{N}_0$ there exists a sequence of bijections $g_n \colon V_n \rightarrow V'_n$ such that $g_n(V_n) = V'_n$ and $f^{(n)}_{vw} = {f'}^{(n)}_{g_{n+1}(v),g_{n}(w)}$. Then we say that the diagrams $B$ and $B'$ are \textit{isomophic with respect to the sequence of vertex bijections $(g_n)$} and denote $B \simeq B'$. 
\end{definition}

If the incidence matrices of $B$ and $B'$ are $0-1$ matrices, then the usual definition of isomorphism and the definition of isomorphism with respect to the sequence of vertex bijections coincide. If, for some $n \in \N_0$, the vertices $v \in V_{n+1}$ and $w \in V_n$ in the diagram $B$ are connected by multiple edges, then we are not concerned with how the set of edges $E(v, w)$ is mapped into the set of edges $E(g_{n+1}(v), g_n(w))$. In other words, if $B \simeq B'$, there may exist multiple isomorphisms mapping $B$ to $B'$. Clearly, $\simeq$ is an equivalence relation.




Below we formulate the fact that the diagrams $B$ and $B'$ are isomorphic in terms of their incidence matrices.
By a \textit{permutation matrix} we will mean a matrix that has exactly one entry of $1$ in each row and each column with all other entries equal to $0$. If $P$ is a permutation matrix, then it is easy to see that $PP^T = P^TP = I$, where $I = (I_{kl})$ is a (finite or infinite) square matrix with $I_{kl} = 1$ for $k = l$ and $I_{kl} = 0$ for $k \neq l$. Clearly, if the diagram $B$ is isomophic to the diagram $B'$ with respect to the sequence of vertex bijections then the following result holds.

\begin{prop}\label{thm:isomBD}
    Let $B= B(F_n)$ and $B' = B'(F'_n)$ be (standard or generalized) Bratteli diagrams. Then $B \simeq B'$ if and only if there exist permutation matrices $P_n$ such that for all $n \in \N_0$:
    $$
    F'_n = P_n F_n P_{n+1}^T.
    $$
\end{prop}

To illustrate Definition \ref{def_isom BD} and Proposition \ref{thm:isomBD}, we consider the following example. This example can also be found in \cite[Example 2.9]{BezuglyiJorgensenKarpelSanadhya2025}. Here, we provide an explicit description of permutation matrices which are used for the isomorphism.


\begin{example}[Isomorphic generalized Bratteli 
diagrams, Figure~\ref{Fig:Isom_BD}]\label{Ex:isom_bd}
Let $B(F)$ be a stationary generalized Bratteli diagram such that 
every level of $B$ is identified with $\Z$, and the incidence 
matrix $F = (f_{ij})_{i,j \in \Z}$ has entries
$$
f_{ij} = 
\left\{
\begin{aligned}
& 2, \mbox{ for } i = j,\\
& 1, \mbox{ for } |i - j| = 1,\\
& 0, \mbox{ otherwise. }
\end{aligned}
\right.
$$ 
Let $B'(F')$ be a stationary generalized Bratteli 
diagram such that every level of $B'$ is identified with $\N_0$, 
and its incidence matrix $F' = (f'_{ij})_{i,j \in \N_{0}}$ has 
entries $f'_{00} = f'_{11} = 2$, $f'_{01} = 1, f'_{02} = 1, f'_{10} = 1, f'_{13} = 1, f'_{20} = 1, f'_{31} = 1$,
and for $i,j \notin \{0,1\}$: 
$$
f'_{ij} = 
\left\{
\begin{aligned}
& 2, \mbox{ for } i = j,\\
& 1, \mbox{ for } |i - j| = 2,\\
& 0, \mbox{ otherwise. }
\end{aligned}
\right.
$$ 
The diagrams $B$ and $B'$ (elaborated in 
Figure~\ref{Fig:Isom_BD}) are isomorphic. Indeed, the bijections 
$(g_n : V_n \rightarrow V_n')_{n \in \N_0}$ and $(h_n : E_n 
\rightarrow E_n')_{n \in \N_0}$ that give the isomorphism are 
defined as follows: since the diagrams are stationary, we set,
for every $n \in \N_0$, $g_n = g : \mathbb{Z} \rightarrow 
\mathbb{N}_0$ 
$$
g (n) = 
\left\{
\begin{aligned}
&2n, \mbox{ if } n \geq 0,\\
&- 2n - 1, \mbox{ if } n < 0.
\end{aligned}
\right.
$$ 
The bijection $h_n : E_n \rightarrow E_n'$ is defined as follows: 
the two vertical edges in the diagram $B(F)$ with range $i \in 
V_n$ are mapped to the two vertical edges in the diagram $B'(F')$ 
with range $g(i) \in V_n$. For $i > 0$, the 
non-vertical edge with range $i$ coming from left (respectively 
from the right) is mapped to the non-vertical edge 
with range 
$g(i) \in V'_n$ coming from the left (respectively from the 
right). For $i < 0$, the 
non-vertical edge with range $i$ coming from left (respectively 
from the right) is mapped to the non-vertical edge 
with range 
$g(i) \in V'_n$ coming from the right (respectively from the 
left). For the non-vertical edges with range $0 \in V_n$, 
the mapping  can be seen from the 
Figure~\ref{Fig:Isom_BD}. It is easy to see that with the above 
bijections the two diagrams are isomorphic.

\begin{figure}
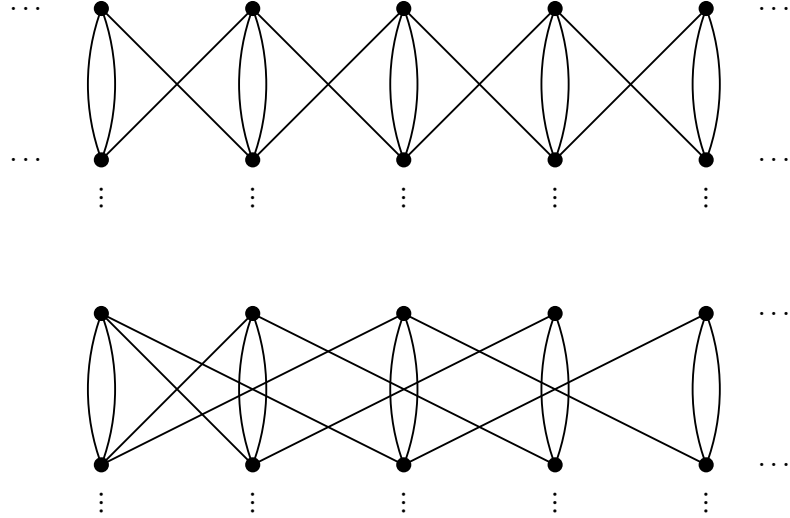

\unitlength=1cm
\begin{graph}(11,4)
 \roundnode{V11}(2,3)
 \roundnode{V12}(4,3)
 \roundnode{V13}(6,3)
 \roundnode{V14}(8,3)
  \roundnode{V15}(10,3)
 \roundnode{V21}(2,1)
 \roundnode{V22}(4,1)
 \roundnode{V23}(6,1)
  \roundnode{V24}(8,1)
    \roundnode{V25}(10,1)
  %
 %
 \graphlinewidth{0.025}

 \bow{V21}{V11}{0.09}
  \bow{V21}{V11}{-0.09}
    \edge{V22}{V11}
    \edge{V21}{V12}
     
 \bow{V22}{V12}{0.09}
 \bow{V22}{V12}{-0.09}

 \bow{V23}{V13}{0.09}
  \bow{V23}{V13}{-0.09}

 \edge{V23}{V12}
    \edge{V22}{V13}
    
     \bow{V24}{V14}{0.09}
  \bow{V24}{V14}{-0.09}
  
   \edge{V24}{V13}
    \edge{V23}{V14}
    
         \bow{V25}{V15}{0.09}
  \bow{V25}{V15}{-0.09}
  
   \edge{V25}{V14}
    \edge{V24}{V15}

 \freetext(10.9,3){$\ldots$}
  \freetext(10.9,1){$\ldots$}
   \freetext(1,3){$\ldots$}
  \freetext(1,1){$\ldots$}
  \freetext(2,0.5){$\vdots$}
  \freetext(4,0.5){$\vdots$}
    \freetext(6,0.5){$\vdots$}
      \freetext(8,0.5){$\vdots$}  
        \freetext(10,0.5){$\vdots$} 
    
\end{graph}

\begin{graph}(11,4)
 \roundnode{V11}(2,3)
 \roundnode{V12}(4,3)
 \roundnode{V13}(6,3)
 \roundnode{V14}(8,3)
  \roundnode{V15}(10,3)
 \roundnode{V21}(2,1)
 \roundnode{V22}(4,1)
 \roundnode{V23}(6,1)
  \roundnode{V24}(8,1)
    \roundnode{V25}(10,1)
  %
 %
 \graphlinewidth{0.025}

 \bow{V21}{V11}{0.09}
  \bow{V21}{V11}{-0.09}
    \edge{V22}{V11}
    \edge{V21}{V12}
     \edge{V21}{V13}
     
 \bow{V22}{V12}{0.09}
 \bow{V22}{V12}{-0.09}

 \bow{V23}{V13}{0.09}
  \bow{V23}{V13}{-0.09}

 \edge{V23}{V11}
    \edge{V22}{V14}
    
     \bow{V24}{V14}{0.09}
  \bow{V24}{V14}{-0.09}
  
   \edge{V24}{V12}
    \edge{V23}{V15}
    
         \bow{V25}{V15}{0.09}
  \bow{V25}{V15}{-0.09}
  
   \edge{V25}{V13}

 \freetext(10.9,3){$\ldots$}
  \freetext(10.9,1){$\ldots$}
  \freetext(2,0.5){$\vdots$}
  \freetext(4,0.5){$\vdots$}
    \freetext(6,0.5){$\vdots$}
      \freetext(8,0.5){$\vdots$}  
        \freetext(10,0.5){$\vdots$} 
    
\end{graph}
\caption{Isomorphic generalized Bratteli diagrams $B$ and 
$B'$.}\label{Fig:Isom_BD}
\end{figure}

We can formulate the fact that the diagrams $B$ and $B'$ are isomorphic in terms of their incidence matrices. Let $P$ be a $\mathbb{N}_0 \times \mathbb{Z}$ matrix such that
$$
p_{ij} = 
\left\{
\begin{aligned}
& 1, \mbox{ for } i = g(j),\\
& 0, \mbox{ otherwise. }
\end{aligned}
\right.
$$ 
It is straightforward to verify that $F' = PFP^T$.
\end{example}





\begin{definition}\label{Def:irreducible_GBD} 
Let $B = B(F_n)$ be
a generalized Bratteli diagram. Assume that all levels $V_i$ of $B$ are identified with a set $V_0$ (e.g. $V_0 = 
\mathbb{N}$ or $\mathbb{Z}$). Then $B$ is called  \textit{irreducible} if 
for any vertices $i, j \in V_0$ and any $n \in \N_0$ there exist 
$m 
> n$ such that there is a finite path between $i \in V_n$ and $j \in V_m$. In 
other words, the $(j, i)$-entry of the matrix $F_{m-1} \cdots 
F_n$ 
is non-zero. Otherwise, the diagram is called \textit{reducible}.
\end{definition}

\begin{remark}
Similarly to the case of finite matrices, a countably infinite matrix $A = (a_{ij})$ is called 
\textit{irreducible} if for every pair of indices $i,j$ there 
is $n > 0$ such that $a^{(n)}_{ij} > 0$, where $a^{(n)}_{ij}$ are elements of the matrix $A^n$. 
A stationary generalized Bratteli diagram is
irreducible if and only if the corresponding incidence matrix is 
irreducible.
Denote $$
p(i) = \gcd\{n:a^{(n)}_{ii} > 0\}.
$$
Then $p(i)$ is called the \textit{period of index} $i$. For an 
irreducible matrix $A$, the periods of all indices are the same 
and called the \textit{period of $A$}. An irreducible matrix with 
period one is called \textit{aperiodic} (see \cite{Kitchens1998} for 
more definitions and results about infinite non-negative matrices).
\end{remark}

\begin{remark}
    We can introduce similar notion of irreducibility for standard Bratteli diagrams, which have the same number of vertices on each level. Such diagrams are diagrams of finite rank (see \cite{BezuglyiKwiatkowskiMedynetsSolomyak2013}). 
\end{remark}

Recall that we identify the vertices of each level $V_n$ of a generalized Bratteli diagram $B$ with the same countably infinite set, usually $\N$ or $\Z$. 
After such identification,  we will call the edges $e_n$ for which $s(e_n) = i \in V_n$ and $r(e_n) = i \in V_{n+1}$ \textit{vertical} and the edges with $s(e_n) = i \in V_n$ and $r(e_n) = j \in V_{n+1}$ with $i \neq j$ \textit{slanted}. 
We call an infinite path \textit{vertical} if all its edges are vertical and \textit{slanted} if all its edges are slanted. It was noticed in \cite[Remark 2.10]{BezuglyiJorgensenKarpelSanadhya2025} that, in general, the property of irreducibility is not preserved under an isomorphism of generalized Bratteli diagrams. For the readers convenience, we repeat the remark below.

\begin{remark}\label{Rem:isomvsirred}\cite{BezuglyiJorgensenKarpelSanadhya2025}
Note that, in general, isomorphism of generalized Bratteli diagrams does not preserve irreducibility. For instance, stationary diagram $B$ from Example~\ref{Ex:isom_bd} is irreducible since starting from every vertex we eventually reach any other vertex for some sufficiently low level. It is easy to see that diagram $B'$ is also irreducible.
    Define another stationary generalized Bratteli diagram 
$B''=(V'', E'')$ isomorphic to $B$ with every level identified with $\Z$. Consider the following sequence of  bijections $g_n \colon V_n \rightarrow V_n''$ between vertices of $B$ and $B''$ on level $n$:
    $$
    g_n(i) = i + n.
    $$
    Thus, the bijections $g_n$ do not change the level $V_0$, shift all the vertices of level $V_1$ by $1$ to the right and all vertices of level $V_n$ by $n$ to the right. 
    For all $n$, the corresponding sequence of bijections $h_n \colon E_n \rightarrow E''_n$ between edges will map all edges outgoing from vertex $i \in V_n$ to the edges with the source $i + n \in V'_n$ and the range 
   in the set $\{i+n, i+n+1, i+n+2\}.$ 
In particular,
the edge with source $i \in V_n$ and range $i - 1 \in V_{n+1}$ will be mapped to the vertical edge between $i+n \in V''_n$ and $i+n \in V''_{n+1}$. Thus, for every $n \in \N_0$, it will not be possible to reach from vertex $i \in V''_n$ any vertex $j < i$ for any level $m > n$. The incidence matrix $F''$ of $B''$ is lower triangular (here we indicate the main diagonal of $F''$ with bold font):
    $$
    F'' = 
    \begin{pmatrix}
       \ddots & \vdots & \vdots & \vdots & \vdots & \udots\\
       \ldots & \textbf{1} & 0 & 0 & 0 &\ldots\\
       \ldots & {2} & \textbf{1} & 0 & 0 & \ldots\\   
       \ldots & {1} & 2 & \textbf{1} & 0 & \ldots\\ 
       \ldots & 0 & 1 & 2 & \textbf{1} & \ldots\\     
       \udots & \vdots & \vdots & \vdots & \vdots & \ddots
    \end{pmatrix}.
    $$
\end{remark}


\section{Classes of generalized Bratteli diagrams}\label{sec:examples}

In this section, we examine in detail several classes of generalized Bratteli diagrams that exhibit a range of behaviours with respect to irreducibility and the topological properties of their tail equivalence relations.
These diagrams also serve as examples illustrating the results presented in Section \ref{sec:mainres}.

\subsection{Generalized Bratteli diagrams of bounded size}\label{subsec:bdd}

Generalized Bratteli diagrams of bounded size  
are used to 
model substitution dynamical systems on infinite alphabets 
(for more information on substitutions on infinite alphabets see, e.g., \cite{Ferenczi_2006}, \cite{DomingosFerencziMessaoudiValle2024}, \cite{Frettloh_Garber_Manibo_2022}, and for the corresponding Bratteli diagram construction see \cite{BezuglyiJorgensenSanadhya2024}). 
Such diagrams have a locally compact path space and can be considered as an intermediate step between 
the standard and generalized Bratteli diagrams. 



\begin{definition}\label{Def:BD_bdd_size} A generalized Bratteli 
diagram $B(F_n)$ is called of \textit{bounded size} if for all $n \geq 0$ the set $V_n$ is identified with $\mathbb{Z}$ and there exists 
a sequence of pairs of natural numbers $(t_n, L_n)_{n \in \N_0}$ 
such that, for all $n \in \mathbb{N}_0$ and all $v \in V_{n+1}$, we have
\begin{equation}\label{eq: Bndd size}
s(r^{-1}(v)) \in \{v - t_n, \ldots, v + t_n\} \quad \mbox{and} 
\quad \sum_{w \in V_{n}} f^{(n)}_{vw} = \sum_{w \in V_{n}} |E(w,v)| 
\leq L_n.
\end{equation} 
\end{definition} 

\begin{remark}
To simplify the notation and calculations, we assume that for each $n \in \N_0$, the natural numbers 
$(t_n, L_n)$ are chosen to be the minimal possible. Also, for every 
$n \in \mathbb{N}_0$, it is assumed that $E(v - t_n, v)$ and 
$E(v + t_n, v)$ are nonempty for all $v \in V_{n+1}$. Throughout this paper, we will not use the bound given by $L_n$ and will be using only the sequence of parameters $(t_n)$.
\end{remark}

\begin{remark}
    In this paper, we will only use the parameters $(t_n)$ for the generalized Bratteli diagrams of bounded size. In general, whether there is a single edge or multiple edges between a given pair of vertices doesn't affect irreducibility or connectedness of a Bratteli diagram, or whether the tail equivalence relation is topologically transitive or minimal. 
\end{remark}

Let $B$ be a generalized Brattlei diagram of bounded size. Since every vertex of $B$ has finitely many outgoing edges, the path space $X_B$ is locally compact. Lemma \ref{lemma_bdd_size_cone}  describes the structure of every generalized Bratteli diagram of bounded size. This result can be found in \cite[Section 3.1]{BezuglyiJorgensenKarpelSanadhya2025}.

\begin{lemma}\cite{BezuglyiJorgensenKarpelSanadhya2025}\label{lemma_bdd_size_cone} Let $B=(V,E)$ be a 
diagram of bounded size corresponding to a sequence 
$(t_n, L_n)_{n \in \N_0}$. Fix $v \in V_{n+1}$. Then 
for all infinite paths $x = 
(x_n)_{n\in \N_0}$ such that $r(x_n) = v$ and all $m > n$ we have
\begin{equation}\label{formula_bdd_size_lower_cone}
r(x_{m}) \in \left\{v - \sum_{i = n+1}^{m} t_{i}, \ldots, v + 
\sum_{i = n+1}^{m} t_{i}\right\},
\end{equation}
and for 
every $k \leq n$ we have 
$$
s(x_k) \subset \left\{v - \sum_{i = k}^{n} t_i, \ldots, v + 
\sum_{i = k}^n t_i\right\}.
$$
\end{lemma}

Let $B$ be a generalized Bratteli diagram of bounded size, $n \in \N_0$ and $v \in V_{n+1}$. We call the set of vertices $\{s(E(V_m, v)): m \leq n\}$ the \textit{upper cone} corresponding to $v$. The set of vertices $\{r(E(V_m, v)): m > n+1\}$ is called the the \textit{lower cone} corresponding to $v$.

For $w \in V_0$, define 
$$
Z_w^+ = \left\{x = (x_n) \in X_B : s(x_0) \geq w \mbox{ and } 
r(x_n) \geq w + \sum_{i = 0}^{n} t_i \mbox{ for all } n \in 
\mathbb{N}_0\right\}.
$$
and
$$
Z_w^- = \left\{x = (x_n) \in X_B : s(x_0) \leq w \mbox{ and } 
r(x_n) \leq w - \sum_{i = 0}^{n} t_i \mbox{ for all } n \in 
\mathbb{N}_0\right\}.
$$
The sets $Z^+_w$, $Z^-_w$ are called \textit{slanting sets} for 
$w \in V_0$. 
The following result can be found in \cite[Lemma 5.3, Theorem 5.4]{BezuglyiJorgensenKarpelSanadhya2025}.



\begin{theorem}\cite{BezuglyiJorgensenKarpelKwiatkowski2025}
    Let $B = (V,E) $ be a generalized Bratteli diagram of bounded size, with the 
corresponding sequence $(t_n, L_n)_{n \in \N_0}$. 
Then, for every $w \in V_0$, the sets $Z^+_w$, $Z^-_w$ are closed 
nowhere dense $\mathcal{R}$-invariant sets. In particular,  the tail equivalence relation 
$\mathcal{R}$ is \textit{not minimal}.
\end{theorem}

The tail equivalence relation for the Bratteli digrams of bounded size can be topologically transitive or not depending on the diagram. The following result was proved in \cite[Theorem 5.1]{BezuglyiJorgensenKarpelSanadhya2025}:

\begin{theorem}\label{ThmBJKStoptrans}\cite{BezuglyiJorgensenKarpelSanadhya2025}
 Let $B = (V,  E)$ be a generalized stationary Bratteli diagram 
with an 
irreducible aperiodic incidence matrix $F = (f_{ij})_{i,j \in 
\mathbb{Z}}$. Then the tail equivalence relation $\mathcal{R}$ is 
topologically transitive.   
\end{theorem}

It was also noticed in \cite[Remark 5.2]{BezuglyiJorgensenKarpelSanadhya2025} that if $B(V, E)$ is an irreducible generalized Bratteli diagram and
there exists an infinite vertical path $x \in X_B$ 
then the orbit $[x]_{\mathcal R}$ is dense in $X_B$, and 
the tail equivalence relation $\mathcal{R}$ is topologically 
transitive. Thus, we obtain the following corollary.

\begin{corol}
Let $B$ be a generalized Bratteli diagram of bounded size and $\mathcal{R}$ be the corresponding tail equivalence relation.

\begin{enumerate}[label=\upshape(\roman*), leftmargin=*, widest=iii]

\item If $B = B(F)$ is stationary and $F$ is irreducible then $\mathcal{R}$ is topologically transitive. In particular, if for some $t \in \N$ the incidence matrix $F$ has the property $f_{ij} > 0$ for $|i-j| \leq t$ then $F$ is irreducible and $\mathcal{R}$ is topologically transitive.

\item Let $B = B(F_n)$ have parameters $(t_n, L_n)$ such that
$f_{ij}^{(n)} > 0$ for all $|i - j| \leq t_n$. Then $B$ is irreducible, has vertical paths, and $\mathcal{R}$ is topologically transitive.
\end{enumerate}
\end{corol}


\begin{example}[irreducible stationary diagram of bounded size with transitive tail equivalence relation]
The two-sided generalized Bratteli diagram $B(F)$ mentioned in Example \ref{Ex:isom_bd} and Remark \ref{Rem:isomvsirred} (see the upper diagram on Figure \ref{Fig:Isom_BD}) is a stationary irreducible generalized Bratteli diagram of bounded size with topologically transitive non-minimal tail equivalence relation. 
\end{example}

\begin{example}[irreducible stationary diagram of bounded size with non-transitive tail equivalence relation]\label{ex:bdd_irred_nontrans}
Let $B$ be a stationary generalized Bratteli diagram of bounded size with the incidence matrix $F = (f_{ij})$ defined as follows:
$$
f_{ij} = 
\left\{
\begin{aligned}
& 1, \mbox{ for } |i - j| = 1,\\
& 0, \mbox{ otherwise. }
\end{aligned}
\right.
$$
Let $\mathcal{R}$ be the tail equivalence relation on $B$.
Then $B$ is irreducible but $\mathcal{R}$ is not topologically transitive. Note that from even vertices on even levels of $B$ we can reach only odd vertices on odd levels of $B$, and vice versa, from odd vertices on even levels, we can reach only even vertices on odd levels. Thus, the diagram $B$ can be represented as a union of two disjoint subdiagrams $\ov B_{1}$ and $\ov B_2$, where the vertices of $\ov B_{1}$ are all even vertices on even levels and all odd vertices on odd levels, and the vertices of $\ov B_2$ are all odd vertices on even levels and all even vertices on odd levels. We have $X_B = X_{\ov B_{1}} \bigsqcup X_{\ov B_2}$, and the path spaces $X_{\ov B_{1}}$ and $X_{\ov B_2}$ are  $\mathcal{R}$-invariant. Since each $X_{\ov B_i}$, $i = 1,2$ is a countable union of cylinder sets of level $0$, both sets $X_{\ov B_{1}}$ and $X_{\ov B_2}$ are clopen.
Hence, the tail equivalence relation on $B$ is not topologically transitive. It is easy to see that the diagram is irreducible: from every vertex $i \in V_0$ (since the diagram $B$ is stationary, it's enough to consider only level $V_0$), we will eventually reach arbitrary vertex $j$ on some level below. The above phenomenon takes place because the matrix $F$ is not aperiodic: it has period $2$ (see e.g. \cite[Chapter 7]{Kitchens1998} or \cite[Appendix A]{BezuglyiJorgensenKarpelSanadhya2025}).
\end{example}

\begin{example}[reducible stationary diagram of bounded size with non-transitive tail equivalence relation]\label{ex:bdd_red_nontrans} We can modify Example \ref{ex:bdd_irred_nontrans} to obtain a reducible generalized Bratteli diagram $\tl B$ of bounded size with non-transitive tail equivalence relation. It suffices to choose $f_{ij} = 1$ for $|i-j| = 2$ and $f_{ij} = 0$ otherwise. Then from even vertices we can only reach even vertices, and from odd vertices we can only reach odd vertices. Thus, such diagram is reducible. The tail equivalence relation is non-transitive for the similar reason as in Example \ref{ex:bdd_irred_nontrans}: the diagram $\tl B$ can be represented as a disjoint union of four subdiagrams.    
\end{example}

\subsection{The renewal shift}\label{subsec:BD_renewal}

Consider the one-sided stationary generalized diagram $B_{RS} = B_{RS}(F_{RS})$, 
where the incidence matrix $F_{RS}$ is defined as follows:
\begin{equation}
    F_{RS} = \begin{pmatrix}
    1 & 1 & 0 & 0 & 0 & 0 & \cdots\\
    1 & 0 & 1 & 0 & 0 & 0 & \cdots\\
    1 & 0 & 0 & 1 & 0 & 0 & \cdots\\
    1 & 0 & 0 & 0 & 1 & 0 & \cdots\\
    1 & 0 & 0 & 0 & 0 & 1 & \cdots\\
    1 & 0 & 0 & 0 & 0 & 0 & \cdots\\
   \vdots & \vdots & \vdots & \vdots & \vdots & \vdots & \ddots
    \end{pmatrix}.
\end{equation}

\begin{figure}[hbt!]
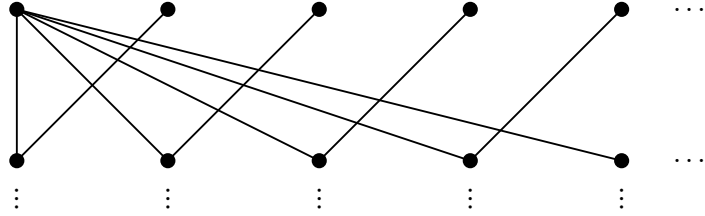

\unitlength=1cm
\begin{graph}(9,4)
  %
   \roundnode{V21}(0,3)
 \roundnode{V22}(2,3)
 \roundnode{V23}(4,3)
 \roundnode{V24}(6,3)
  \roundnode{V25}(8,3)
 \roundnode{V31}(0,1)
 \roundnode{V32}(2,1)
 \roundnode{V33}(4,1)
  \roundnode{V34}(6,1)
    \roundnode{V35}(8,1)
  
 %
 \graphlinewidth{0.025}


    \edge{V31}{V21}
        \edge{V31}{V22}
      
      \edge{V32}{V21} 
   \edge{V32}{V23}
     
 \edge{V33}{V21}
  \edge{V33}{V24}

   \edge{V34}{V21} 
    \edge{V34}{V25}

   \edge{V35}{V21}

 \freetext(8.9,3){$\ldots$}
  \freetext(8.9,1){$\ldots$}
  
  \freetext(0,0.5){$\vdots$}
  \freetext(2,0.5){$\vdots$}
    \freetext(4,0.5){$\vdots$}
      \freetext(6,0.5){$\vdots$}  
        \freetext(8,0.5){$\vdots$}

\end{graph}
\caption{One-sided stationary generalized Bratteli diagram $B_{RS}$ which corresponds to the renewal shift.}\label{Fig:BDrenewal}
\end{figure}

In other words, the vertex $1$ of each level is connected by a single edge to all vertices on the level below, and for $n > 1$, the vertex $n$ is connected by a singe edge to the vertex $(n-1)$ on the level below (see Figure \ref{Fig:BDrenewal}). 
The matrix $A_{RS} = F_{RS}^T$ appears in the theory of countable Markov shifts and corresponds to the renewal shift (see \cite[Chapter 7]{Kitchens1998} and \cite[Subsection 4.3.1]{Raszeja2021}, \cite[Subsection 3.2]{BEFR2022}). 
The diagram $B_{RS}$ was considered in \cite[Example 7.16]{BezuglyiJorgensenKarpelSanadhya2025}.
This diagram possesses a unique ergodic invariant probability measure which takes positive values on all cylinder sets (see \cite[Theorem 7.2, Proposition 7.17]{BezuglyiJorgensenKarpelSanadhya2025}). One can also find modifications of the diagram $B_{RS}$ and the corresponding incidence matrix, e.g., in \cite[Subection 7.3]{BezuglyiJorgensenKarpelSanadhya2025} and \cite[Subection 4.3]{Raszeja2021}.

Theorem \ref{thm:minimalRrenewalshift} below shows that the diagram $B_{RS}$ provides an example of an irreducible generalized Bratteli diagram with a minimal tail equivalence relation. 

\begin{theorem}\label{thm:minimalRrenewalshift}
    Let $\mathcal{R}_{RS}$ be the tail equivalence relation on the generalized Bratteli diagram $B_{RS}$ which corresponds to the renewal shift. Then $\mathcal{R}_{RS}$ is minimal and $B_{RS}$ is irreducible.
\end{theorem} 

\begin{proof} First, we show that $B_{RS}$ is irreducible.
From the vertex $i = 1$ on any level of the diagram $B_{RS}$, there are edges to any vertex $j \in \N$ on a level below. 
For every $n \in \N_0$ and any vertex $i > 1$ in $V_n$, there is a finite path from $i$ to the first vertex on the level $n + i - 1$. Hence, from the vertex $i$ on the level $n$, we can reach any vertex $j$ on the level $n + i$. Thus, the diagram $B_{RS}$ is irreducible.

Now we show that $\mathcal{R}_{RS}$ is minimal. Due to the structure of $B_{RS}$ (see Figure \ref{Fig:BDrenewal}), any path from $X_{B_{RS}}$ passes through the first vertex infinitely many times. Let $x \in X_{B_{RS}}$ be an arbitrary path and $[(y_0, \ldots, y_n)]$ be an arbitrary cylinder set. Then $r(y_n) \in V_{n+1}$ is joined by a finite path to the first vertex on some level $m > n+1$. 
Since $x$ passes through the first vertex infinitely many times, there is a level $m_1 > m$ such that $s(x_{m_1}) = 1$. There is a vertical path in $B$ between the first vertex on level $m$ and the first vertex on level $m_1$, hence, there is an infinite path $z = (z_i)$ such that $z_i = x_i$ for $i \geq m_1$, $s(z_i) = 1$ for $m \leq i \leq m_1$ and $z_i = y_i$ for $0 \leq i \leq n$. Thus, $z$ is tail equivalent to $x$ and belongs to the set $[(y_0, \ldots, y_n)]$. In other words, the orbit of $x$ under the tail equivalence relation visits $[(y_0, \ldots, y_n)]$. Thus, the orbit
of every infinite path under the tail equivalence relation visits every cylinder set, thus the tail equivalence relation $\mathcal{R}_{RS}$ is minimal.  
\end{proof}

\subsection{Diagrams of infinite odometers}\label{subsec:DIO}
A class of generalized Bratteli diagrams called \textit{diagrams of infinite odometers} ($B_{IO}$) was introduced in \cite[Section 4]{BezuglyiKarpelKwiatkowski2024}. It's incidence matrices are $\N \times \N$ upper triangular matrices
$$
{F}_{n} =
\begin{pmatrix}
 a_{1}^{(n)} &     1 &      0 &      0 & \ldots  \\
     0 &  a_{2}^{(n)} &      1 &      0 &      \ldots  \\
     0 &      0 &  a_{3}^{(n)} &      1 &      \ldots  \\
     0 &      0 &      0 &  a_{4}^{(n)} &      \ldots  \\
\vdots & \vdots & \vdots & \vdots & \ddots \\
\end{pmatrix}, \quad n \in \N_0,
$$
with entries
$$
f_{ij}^{(n)} = 
\left\{
\begin{aligned}
& a_{i}^{(n)} \geq 2, \mbox{ for } j = i,\\
& 1, \mbox{ for } j = i + 1,\\
& 0, \mbox{ otherwise. }
\end{aligned}
\right.
$$

By \textit{odometer} we will mean a standard Bratteli diagram which has a single vertex on each level and at least two edges connecting the vertices of neighbouring levels (see \cite[Subsection 6.5.1]{Durand2010}). For $i \in \N$, denote by $B_i$ the vertical odometer which passes through the vertex $i$ on each level. In other words, odometer $B_i$ has $a_n^{(i)}$ edges connecting the vertex of level $n$ with the vertex of level $n + 1$. Denote by $X_{B_i}$ the set of all paths from $X_B$ that pass through vertex $i$ on each level. Denote by $\wh {X}_{B_i}$ the set of all paths from $X_B$ that are tail equivalent to a path from $X_{B_i}$. Then $X_{B_{IO}} = \bigsqcup_{i = 1}^{\infty}\wh X_{\ov B_i}$ (see \cite[Proposition 4.1]{BezuglyiKarpelKwiatkowski2024}). In other words, every infinite path is eventually vertical and is tail equivalent to a path from $B_i$ for some $i \in \N$.

For diagram $B_{IO}$, one can characterize the set of all ergodic invariant measures (see \cite[Theorem 4.3]{BezuglyiKarpelKwiatkowski2024}). In particular, diagrams of infinite odometers provide examples of generalized Bratteli diagrams which possess no finite ergodic invariant measures. For instance, this happens for $a_{i}^{(n)} = i + 1$ for all $n \in \N_0$, see Figure \ref{Fig:no measure}.

\begin{figure}[hbt!]
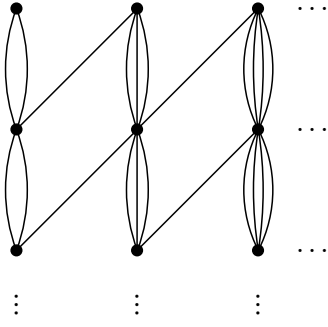

\unitlength=.8cm
\begin{graph}(7,6)
 \roundnode{V11}(2,5)
 \roundnode{V12}(4,5)
 \roundnode{V13}(6,5) 
 \roundnode{V21}(2,3)
 \roundnode{V22}(4,3)
 \roundnode{V23}(6,3)
 \roundnode{V31}(2,1)
 \roundnode{V32}(4,1)
 \roundnode{V33}(6,1)
  %
 %
 \graphlinewidth{0.025}

 \bow{V21}{V11}{0.09}
  \bow{V21}{V11}{-0.09}
    \edge{V21}{V12}
     
 \bow{V22}{V12}{0.09}
 \bow{V22}{V12}{-0.09}
 \edge{V22}{V12}
 
 \bow{V23}{V13}{0.12}
  \bow{V23}{V13}{-0.12}
   \bow{V23}{V13}{0.04}
  \bow{V23}{V13}{-0.04}
 \edge{V22}{V13}

 \bow{V31}{V21}{0.09}
 \bow{V31}{V21}{-0.09}
   \edge{V31}{V22}
 \bow{V32}{V22}{0.09}
 \bow{V32}{V22}{-0.09}
 \edge{V32}{V22}

  \bow{V33}{V23}{0.12}
  \bow{V33}{V23}{-0.12}
   \bow{V33}{V23}{0.04}
  \bow{V33}{V23}{-0.04}
 \edge{V32}{V23}
 
    \freetext(6.9,5){$\ldots$}
  \freetext(6.9,3){$\ldots$}
    \freetext(6.9,1){$\ldots$}
    \freetext(2,0.1){$\vdots$}
  \freetext(4,0.1){$\vdots$}
    \freetext(6,0.1){$\vdots$}
\end{graph}
\caption{A stationary one-sided $B_{IO}$ that does not have finite invariant measures.}\label{Fig:no measure}
\end{figure}

In this paper, we will call diagrams $B_{IO}$ \textit{one-sided diagrams of infinite odometers}. We also introduce \textit{two-sided Bratteli diagrams of infinite odometers} $\tl B_{IO}$ with $\Z \times \Z$ incidence matrices $\tl F_n$ that have entries $a_i^{(n)} \geq 2$ on the main diagonal and ones on the diagonal above:
    $$
    \tl F_n = 
    \begin{pmatrix}
       \ddots & \vdots & \vdots & \vdots & \vdots & \udots\\
       \ldots & \textbf{$a_{-1}^{(n)}$} & 1 & 0 & 0 &\ldots\\
       \ldots & {0} & \textbf{$a_0^{(n)}$} & 1 & 0 & \ldots\\   
       \ldots & {0} & 0 & \textbf{$a_1^{(n)}$} & 1 & \ldots\\ 
       \ldots & 0 & 0 & 0 & \textbf{$a_{2}^{(n)}$} & \ldots\\     
       \udots & \vdots & \vdots & \vdots & \vdots & \ddots
    \end{pmatrix}.
    $$


The structure of the two-sided diagram $\tl B_{IO}$ is more complicated than the one of $B_{IO}$ since not every path in $X_{\tl B_{IO}}$ is eventually vertical. Since $\tl F_n$ has only values $1$ right above the main diagonal, the set of slanted paths is countable (see also \cite[Proposition 5.5]{BezuglyiJorgensenKarpelSanadhya2025}).

\begin{prop}
Both one-sided and two-sided Bratteli diagrams of infinite odometers $B_{IO}$ and $\tl B_{IO}$ are reducible, their corresponding tail equivalence relations $\mathcal{R}_{IO}$ and $\widetilde{\mathcal{R}}_{IO}$ are topologically transitive but not minimal. 
\end{prop}

\begin{proof}
For both diagrams $B_{IO}$ and $\tl B_{IO}$, from a vertex $i$ on level $n$ we will never reach vertex $j > i$ on any level below. Hence, both diagrams $B_{IO}$ and $\tl B_{IO}$ are reducible.

For the one-sided diagram $B_{IO}$, 
the set $\wh{X}_{B_{RS}}$ is dense in $X_{B_{IO}}$. The orbit a path $x \in X_{B_{IO}}$ is dense in $X_{B_{IO}}$ if and only if it $x$ is tail equivalent to a path from $X_{B_{RS}}$. Thus, the tail equivalence relation $\mathcal{R}_{IO}$ is topologically transitive but not minimal.


For the two-sided diagram of infinite odometers $\tl B_{IO}$, there is no $i \in \Z$ such that $\wh{X}_{B_i}$ is dense in $X_{\tl B_{IO}}$. The tail equivalence class of any eventually vertical path is not dense, but we can find other paths that have dense orbits under the tail equivalence relation on the path space of $\tl B_{IO}$. 
First, note that any path $x \in X_{\tl B_{IO}}$ divides the set of all vertices of $\tl B_{IO}$ into three disjoint nonempty subsets: the set $V(x) = \{s(x_n)\}_{n = 0}^{\infty}$ of vertices through which $x$ passes, the set 
$$V_{l}(x) = \bigcup_{n \geq 0}\{w \in V_n : w < s(x_n)\}$$ 
and the set 
$$V_{r}(x) = \bigcup_{n \geq 0}\{w \in V_n : w > s(x_n)\}.$$ 

Clearly, a path with a dense orbit cannot be vertical. Indeed, the paths from the tail equivalence class of a vertical path passing through the vertex $i$ on each level can reach only vertices $j \geq i$ on the levels above, i.e. the vertices from $V_r(x)$. Thus, they will never belong to a cylinder set which ends in a vertex from $V_l(x)$. In particular, they will never visit a cylinder set which corresponds to the vertex $i - 1$ on level $0$. 
     The slanted paths which start at some vertex $i \in V_0 = \Z$ and go all the time through the slanted edges also don't have dense orbits. Indeed, orbit of such a path will never reach a cylinder set which corresponds to the vertex $i+1$ on level $0$ or, in general, any other vertex from $V_r(x)$.
     
     The path $x$ which starts at a vertex $i \in V_0$ and passes in turn through vertical and slanted edges has a dense orbit. In other words, we consider the path which passes through the vertices $i, i, i - 1, i - 1, \ldots, i - k, i - k, \ldots$. 
     For simplicity, set $i = 0$, then we have $s(x_{2n}) = s(x_{2n + 1}) = -n$. Thus, $x$ passes through the vertex number $-n$ on levels $2n+1$ and $2n+2$. 
      Let $[(y_0, \ldots, y_n)]$ be any cylinder set that ends at the vertex $j^{(n+1)}$ on the level $n + 1$. If $j^{(n+1)} \in V_{l}(x)$, then we can use a vertical path that starts at $j^{(n+1)}$ to reach a vertex from $V(x)$. Thus, a point from the orbit of $x$ will belong to the cylinder $[(y_0, \ldots, y_m)]$. If $j^{(m+1)} \in V_{r}(x)$, we use the slanted path $sl(j^{(m+1)})$ that starts at $j^{(m+1)}$ to reach a vertex from $V(x)$. Such a path passes through the vertex $j^{(m+1)} - k$ on the level $m + 1 + k$.  If $j^{(m+1)} > 0$ then the path $sl(j^{(m+1)})$ passes through the vertex $0$ on the level $m + 1 + j^{(m+1)}$ and it passes though vertex $-n$ on the level $m + 1 + j^{(m+1)} + n$. Thus, for $n = j^{(m+1)} + m$ and $n = j^{(m+1)} + m - 1$, the slanted path $sl(j^{(m+1)})$ passes through the vertex which belongs to $V(x)$. If $j^{(m+1)} \leq 0$, we shift all the vertices uniformly by the same number $-r(x_m)$ to the right, so that now we have $r(x_m) = 0$, and consider only the part of $\tl B_{IO}$ which starts at the level $m+1$, Thus, we repeat the same proof as in the previous case, i.e. the proof for $j^{(m+1)} > 0$.  

      Note that the same proof works for any path $x \in X_{\tl B_{IO}}$ has infinitely many vertical and infinitely many slanted edges. In other words, $x \in X_{\tl B_{IO}}$ has a dense orbit if and only if $x$ is not tail equivalent to a vertical or to a slanted path. Such a point $x$ will ``intersect'' any slanted path that starts at a vertex from $V_r(x)$ and any vertical path that starts at a vertex from $V_{l}(x)$. Since $\tl B_{IO}$ contains only edges which are vertical or slanted that join a vertex $i$ to a vertex $i - 1$ on a level below, any path which passes through vertices from both $V_r(x)$ and $V_l(x)$ should pass through a vertex from $V(x)$. 
\end{proof}

\subsection{Diagram $B_{\infty}$}\label{subsec:diagramB-infty}
Diagram $B_{\infty}$ is a one-sided generalized Bratteli diagram defined by the lower triangular $\N \times \N$ incidence matrix $F_{\infty}$ which has all entries equal to $1$ on the main diagonal and below the main diagonal (see Figure \ref{Fig:Binfty}):

$$
F_{\infty} = 
\begin{pmatrix}
    1 & 0 & 0 & 0 & \ldots\\
    1 & 1 & 0 & 0 & \ldots\\
    1 & 1 & 1 & 0 & \ldots\\
    1 & 1 & 1 & 1 & \ldots\\
    \vdots & \vdots & \vdots & \vdots & \ddots
    \end{pmatrix}
$$

\begin{figure}[hbt!]
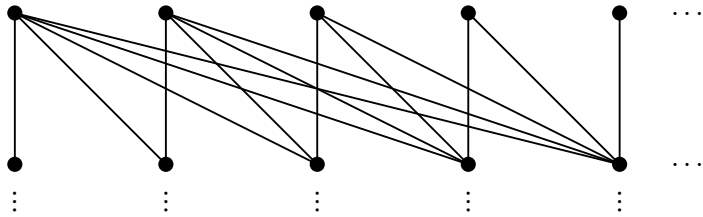

\unitlength=1cm
\begin{graph}(9,4)
  %
   \roundnode{V21}(0,3)
 \roundnode{V22}(2,3)
 \roundnode{V23}(4,3)
 \roundnode{V24}(6,3)
  \roundnode{V25}(8,3)
 \roundnode{V31}(0,1)
 \roundnode{V32}(2,1)
 \roundnode{V33}(4,1)
  \roundnode{V34}(6,1)
    \roundnode{V35}(8,1)
  
 %
 \graphlinewidth{0.025}


    \edge{V31}{V21}
    
    \edge{V32}{V21}
        \edge{V32}{V22}
     
 \edge{V33}{V21}
  \edge{V33}{V22}
 \edge{V33}{V23}
    
   \edge{V34}{V21} 
    \edge{V34}{V22} 
       \edge{V34}{V23} 
   \edge{V34}{V24} 
   
   \edge{V35}{V21}
      \edge{V35}{V22}
         \edge{V35}{V23}
            \edge{V35}{V24}
   \edge{V35}{V25}

 \freetext(8.9,3){$\ldots$}
  \freetext(8.9,1){$\ldots$}
  
  \freetext(0,0.5){$\vdots$}
  \freetext(2,0.5){$\vdots$}
    \freetext(4,0.5){$\vdots$}
      \freetext(6,0.5){$\vdots$}  
        \freetext(8,0.5){$\vdots$}

\end{graph}
\caption{One-sided stationary generalized Bratteil diagram $B_{\infty}$.}\label{Fig:Binfty}
\end{figure}

\smallskip

The diagram $B_{\infty}$ was introduced in \cite[Section 8]{BezuglyiKarpelKwiatkowskiWata2024} as an example of a diagram which has uncountably many probability ergodic invariant measures, and one can explicitly describe these measures using the eigenvectors and eigenvalues of the incidence matrix (see \cite[Theorem 4.3]{BezuglyiKarpelKwiatkowski2024}). The diagram is reducible, since from every vertex number $i > 1$ on a level $n \in \N_0$, we cannot reach a vertex number $j < i$ on a lower level. Proposition \ref{prop:B_infty} shows that the tail equivalence relation on $X_{B_{\infty}}$ is topologically transitive but not minimal. 
\begin{prop}\label{prop:B_infty}
Let $x = (x_n)$ be a path in $X_{B_{\infty}}$. Then the orbit $\mathcal{O}(x)$ is dense in $X_B$ if and only if the set $\{s(x_n)\}_{i = 0}^{\infty} \subset \N$ is unbounded.
\end{prop}

\begin{proof}
Fist, assume that we have $s(x_n) < M$ for some $M \in \N$ and all $n \in \N$. Notice that from a vertex $i$ on a level $n \in \N$ we can reach only vertices with numbers $j \leq i$ on any level above. Hence, $\mathcal{O}(x) \cap [(y_0, \ldots, y_n)] = \emptyset$ for every cylinder set $[(y_0, \ldots, y_n)]$ with $r(y_n) \geq M$.

Now assume that for every $M \in \N$ there exists $n \in \N$ such that $s(x_n) > M$. 
It follows from the structure of the diagram that $s(x_m) \geq s(x_n) > M$ for all $m \geq n$. Let $[(y_0, \ldots, y_n)]$ be a cylinder set. Then there exists $m \in \N$ such that $s(x_m) > r(y_n)$ and $m > n + 1$. Thus, from a vertex $s(x_m) \in V_m$, we can reach a vertex number $s(x_m)$ on level $V_{n+2}$, and since $s(x_m) > r(y_n)$, there is an edge between the vertex $r(y_n) \in V_{n+1}$ and the vertex number $s(x_m)$ on level $V_{n+2}$. Thus, there is a finite path between $r(y_n)$ and $s(x_m)$, and $\mathcal{O}(x) \cap [(y_0, \ldots, y_n)] \neq \emptyset$.
\end{proof}

\section{Isomorphic generalized Bratteli diagrams}\label{sec:mainres}

\subsection{Properties invariant under isomorphism}

Let $B$ and $B'$ be (standard or generalized) Bratteli diagrams with the corresponding tail equivalence relations $\mathcal{R}$ and $\mathcal{R}'$.
Assume that $B$ is isomorphic to $B'$ via the sequences of vertex bijections $(g_n)_{n = 0}^{\infty}$ and edge bijections $(h_n)_{n = 0}^{\infty}$. It is straightforward to check that the map $\theta: X_B \rightarrow X_{B'}$ defined by $$\theta((x_n)_{n = 0}^{\infty}) = (h_n(x_n))_{n = 0}^{\infty}$$ is a homeomorphism which preserves tail equivalence relation. In other words, for all $x,y \in X_B$, we have $(x,y) \in \mathcal{R}$ if and only if $(\theta(x), \theta(y)) \in \mathcal{R}'$. Thus, all topological properties of the tail equivalence relation such as minimality or topological transitivity 
are preserved under the isomorphism of (standard or generalized) Bratteli diagrams. It also follows that $\mathcal{R}$ and $\mathcal{R}'$ are Borel isomorphic (see \cite{Kechris2024} for more details).

A matrix $F = (f_{ij})$ satisfies the \textit{equal row sum} property (we write $F\in ERS$ or $F\in ERS(r)$) if there exists $r$ such that $\sum_{j} f_{ij} = r$ for all $i$. A matrix $F = (f_{ij})$ has the {\em equal column sum} property ($F \in ECS$ or $F\in ECS(c)$) if there exists $c$ such that $\sum_{i} f_{ij} = c$ for all $j$. We say that a Bratteli diagram $B$ has the $ERS(r_n)$ ($ECS(c_n)$) property, if for all $n$ the incidence matrix $F_n \in ERS(r_n)$ ($ECS(c_n)$). Since isomorphism of Bratteli diagrams preserves levels, it preserves the properties of incidence matrices to have equal row (equal column) sum.

Hence, we obtain the following Proposition \ref{Prop:prop_pres_under_isom}.

\begin{prop}\label{Prop:prop_pres_under_isom}
    Let $B$ and $B'$ be two isomorphic (standard or generalized) Bratteli diagrams and $\mathcal{R}$, $\mathcal{R}'$ be the corresponding tail equivalence relations. Then
\begin{enumerate}[label=\upshape(\roman*), leftmargin=*, widest=iii]

\item If $B$ has the $ERS(r_n)$, $n \geq 0$ ($ECS(c_n)$, $n \geq 0$) property then $\tl B$ has the $ERS(r_n)$, $n \geq 0$ ($ECS(c_{n})$, $n \geq 0$) property.

\item  If $\mathcal{R}$ is topologically transitive (minimal) then $\mathcal{R'}$ is topologically transitive (minimal).

\end{enumerate}
\end{prop}

\begin{definition}
    A generalized Bratteli diagram $B$ is called \textit{connected} if it is connected as an undirected graph.
\end{definition}

Clearly, connectedness as an undirected graph is also preserved under isomorphism. The properties of being stationary or of bounded size are not preserved under isomorphism. Below we show that, in general, irreducibility is also not preserved under isomorphism. 

\subsection{Irreducibility and isomorphism}


In this subsection, we  show that the set of all generalized Bratteli diagrams can be partitioned into two disjoint subsets:

\begin{enumerate} [label=\upshape(\roman*), leftmargin=*, widest=iii]
    \item The diagrams which can be isomorphic only to irreducible generalized Bratteli diagrams. We  call such diagrams \textit{completely irreducible}.

    \item The diagrams which can be isomorphic both to reducible and to irreducible  generalized Bratteli diagrams. We call such diagrams \textit{relatively irreducible}.
    

\end{enumerate}

Theorem \ref{thm:classII} below shows that there are no generalized Bratteli diagrams that are isomorphic only to reducible generalized Bratteli diagrams. It other words, the vertices of every generalized Bratteli diagram can be enumerated in such a way that the diagram becomes an irreducible generalized Bratteli diagram.

First, note that the following result  holds.

\begin{prop}
\label{prop:toptransirred}
     Let $B$ be a generalized Bratteli diagram and $\mathcal{R}$ be the tail equivalence relation on $B$. If $\mathcal{R}$ is topologically transitive then $B$ is isomorphic to an irreducible generalized Bratteli diagram.
\end{prop}

\begin{proof}
    Let $x = (x_n) \in X_B$ be a point with dense orbit. Consider a diagram $B'=(V',E')$ isomorphic to $B$ such that the image $\psi(x)$ of $x$ under the isomorphism passes through the vertex number $n$ infinitely many times for every $n$. For instance, the vertices $g(s(x_n))$ might form the following sequence: $0, 0, 1, 0, 1, 2, 0, 1, 2, 3,\ldots$ Since $\mathcal{O}(x)$ is dense in $X_B$, for every cylinder set $[(e_0, \ldots, e_n)] \subset X_B$ there is a finite path between $r(e_n)$ and a vertex $s(x_m)$ for some $m$. Hence, for every cylinder set $[(e_0', \ldots, e_n')] \subset X_{B'}$ there is a finite path from $r'(e_n')$ to $s'(h(x_m)) = g_m(s(x_m))$ for some $m$. Since any number $j \in \N_0$ appears among $\{g_k(s(x_k))\}_{k \in \N_0}$ infinitely many times, we will reach any vertex $j$ from $g_m(s(x_m))$. Thus, the diagram $B'$ is irreducible.
\end{proof}


For instance, diagram $B_{\infty}$ (see Subsection \ref{subsec:diagramB-infty}) and diagrams of infinite odometers (see Subsection \ref{subsec:DIO}) have topologically transitive tail equivalence relations. Below we give an example of a generalized Bratteli diagram with a non topologically transitive tail equivalence relation. Moreover, this diagram 
has countably infinitely many distinct sets $\ov{\mathcal{O}(x^i)}$ such that $X_B = \bigcup_{i = 1}^{\infty} \ov{\mathcal{O}(x^i)}$.

\begin{example}\label{ex:connected_non-transR}
The one-sided stationary generalized Bratteli diagram $B$ shown on Figure \ref{Fig:nontransitnonloccomp} with vertices of each level enumerated by $\N$, has the incidence matrix $F = (f_{ij})$ defined as follows: $f_{11} = 2$; $f_{ii} = 3$ for $i > 1$; $f_{i1} = 1$ for $i > 1$; and $f_{ij} = 0$ otherwise.
The diagram has a unique minimal component for the tail equivalence relation: it is the dyadic odometer which passes through the first vertex on each level of the diagram. For $i \in \N$, let $x^i$ be a path in $X_B$ which eventually passes through the vertex number $i$ on each level of the diagram. Then $X_B = \bigcup_{i = 1}^{\infty} \ov{\mathcal{O}(x^i)}$ and $\ov{\mathcal{O}(x^i)} \cap \ov{\mathcal{O}(x^j)} = \ov{\mathcal{O}(x^1)}$ for all $i \neq j$. Note that all paths in $X_B$ are eventually vertical, and there is no path in $X_B$ with a dense orbit. Thus, the tail equivalence relation on $X_B$ is not topologically transitive.

\begin{figure}
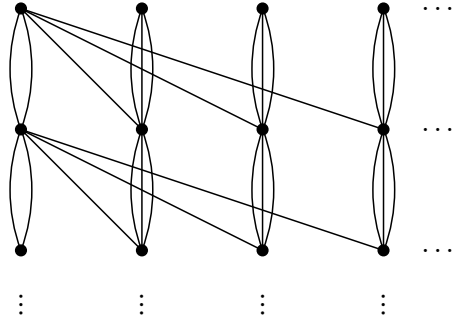

    \unitlength=.8cm
\begin{graph}(9,6)
 \roundnode{V11}(2,5)
 \roundnode{V12}(4,5)
 \roundnode{V13}(6,5) 
  \roundnode{V14}(8,5)
 \roundnode{V21}(2,3)
 \roundnode{V22}(4,3)
 \roundnode{V23}(6,3)
  \roundnode{V24}(8,3)
 \roundnode{V31}(2,1)
 \roundnode{V32}(4,1)
 \roundnode{V33}(6,1)
  \roundnode{V34}(8,1)
  %
 %
 \graphlinewidth{0.025}

 \bow{V21}{V11}{0.09}
  \bow{V21}{V11}{-0.09}
    \edge{V22}{V11}
     
 \bow{V22}{V12}{0.09}
 \bow{V22}{V12}{-0.09}
 \edge{V22}{V12}
 
 \edge{V23}{V13}
 \bow{V23}{V13}{0.09}
 \bow{V23}{V13}{-0.09}
 \edge{V23}{V11}

 \edge{V24}{V14}
 \bow{V24}{V14}{0.09}
 \bow{V24}{V14}{-0.09}
 \edge{V24}{V11}

 \bow{V31}{V21}{0.09}
 \bow{V31}{V21}{-0.09}
   \edge{V32}{V21}
 \bow{V32}{V22}{0.09}
 \bow{V32}{V22}{-0.09}
 \edge{V32}{V22}

 \bow{V33}{V23}{0.09}
 \bow{V33}{V23}{-0.09}
 \edge{V33}{V21}
  \edge{V33}{V23}

 \bow{V34}{V24}{0.09}
 \bow{V34}{V24}{-0.09}
 \edge{V34}{V21}
  \edge{V34}{V24}  
 
    \freetext(8.9,5){$\ldots$}
  \freetext(8.9,3){$\ldots$}
    \freetext(8.9,1){$\ldots$}
    \freetext(2,0.1){$\vdots$}
  \freetext(4,0.1){$\vdots$}
    \freetext(6,0.1){$\vdots$}
    \freetext(8,0.1){$\vdots$}
\end{graph}
\caption{A stationary diagram with non-transitive tail-equivalence relation.}\label{Fig:nontransitnonloccomp}
\end{figure}
\end{example}

We generalize the Proposition~\ref{prop:toptransirred} to obtain the following theorem:

\begin{thm}\label{thm:classII}
    Every generalized Bratteli diagram is isomorphic to an irreducible generalized Bratteli diagram. 
\end{thm}

\begin{proof} 
    Let $B$ be a generalized Bratteli diagram. Then $X_B$ is a Polish space, hence, there exist a countable set of paths $\{x^{i}\}_{i = 0}^{\infty} \subset X_B$ which is dense in $X_B$. We show that there exists an enumeration of the vertices of $B$ such that for all $n$, the set $V_n$ is identified with $\N_0$, and all the paths $\{x^{i}\}_{i = 0}^{\infty}$ pass through the vertex number $j$ infinitely many times for every $j \in \N_0$.
We use the construction which is similar to the construction of Toeplitz sequences (see, e.g., \cite{Kurka2003} or \cite{Downarowicz2005}). Enumerate the vertices of $B$ so that the path $x^0$ passes through the vertex number $0$ on level $0$. Then we do not specify through which vertex  $x^0$ passes on level $1$. Then let $x^0$ pass  through vertex $0$ on level $2$ and through vertex $1$ on level $3$. We do not specify though which vertices $x^0$ passes on levels $4$ and $5$. Let $x^0$ pass through vertex $0$ on level $6$, through vertex $1$ on level $7$, and vertex $2$ on level $8$. 
We do not specify though which vertices $x^0$ passes on levels $9,10,11$, 
and so on. In other words, we consider sequence of vertices through which the path $x^0$ passes. Starting with the first vertex $s(x^0)$ and $n = 1$, we enumerate consecutive blocks of vertices of length $n$ by the numbers $0, \ldots, n-1$, and we leave the next block of vertices of length $n$ without forcing any enumeration. We show the procedure below. We denote by $*$ the vertices of $x^0$ without specified enumeration. The positions of the numbers in the sequence below correspond to the levels where we enumerate the vertices of $x^0$ by the given numbers:
\begin{equation}\label{x^0}
0*01**012***0123****\ldots
\end{equation}
Now we enumerate some of the vertices of the path $x^1$. We will enumerate only those vertices which belong to the levels, for which we didn't specify enumeration for $x^0$. In such a way, we enumerate the vertices of $x^0$ and $x^1$ on independent levels. Moreover, among the blocks of vertices, which are not enumerated for $x^0$, we keep every second block of vertices not enumerated for $x^1$.  
In other words, for $x^1$, we enumerate vertices only from the levels that correspond to the blocks of stars $*$ of odd lengths. In those blocks of odd lengths $2n-1$, we fill in only initial positions of lengths $n$ by the symbols $0, \ldots, n-1$. Hence, $x^1$ passes through vertex $0$ on level $1$, then through vertex $0$ on level $9$, vertex $1$ on level $10$, and so on. We show the enumerated vertices for $x^1$ on the corresponding levels with barred symbols:
$$
0\ov{0}01**012\ov{0}\ov{1}*0123****\ldots
$$
We don't use anymore in the construction the levels that corresponded to the blocks of stars of odd lengths in \eqref{x^0}, even if we didn't fully fill them in. For instance, we will never specify the number of vertex on level $11$ for any path $x^i$, where $i \in \N_0$. 
Now we enumerate in a similar way some of the vertices of the path $x^2$, which belong to the levels, where we haven't enumerated vertices of $x^0$ and $x^1$. Again, among the blocks of stars from \eqref{x^0} of even lengths, we leave every second block not enumerated. The numbers of vertices for $x^2$ on the corresponding levels are shown below by underlined symbols:
$$
0\ov{0}01\underline{0}*012\ov{0}\ov{1}*0123****\ldots
$$
We repeat the procedure for each $x^i$, $i \in \N_0$. Thus, every path $x^i$ will pass through every vertex $j \in \N_0$ infinitely many times.

Now we show that with the enumeration above, the diagram $B$ becomes irreducible.
Let $n \in \N_0$ and $i,j \in \N_0$. Consider any cylinder set $[y_0, \ldots, y_n]$ such that $r(y_n) = i$. 
Then there exists $k \in \N$ such that $x^k = (x^k_l)_{l=0}^{\infty} \in [y_0, \ldots, y_n]$. In other words, $x^k_l = y_l$ for $l = 0, \ldots, n$. 
Since $x^k$ passes through all vertices number $j$ infinitely many times for every $j \in \N_0$, there exists level $m > n+1$ such that $s(x^k_m) = j$. Thus, there is a finite path between a vertex $i$ on level $n$ and a vertex $j$ on some level $m > n$.
\end{proof}

\begin{remark}
Note that Theorem \ref{thm:classII} also holds for standard Bratteli diagrams with the same number of vertices on each level, with a similar proof.
\end{remark}


Below we provide examples of completely and relatively irreducible generalized Bratteli diagrams, and give necessary conditions and sufficient conditions for a diagram to be completely or relatively irreducible.

\begin{example}[A completely irreducible generalized Bratteli diagram]\label{ex:GBDclass1} The Bratteli diagram $B_{RS}$ corresponding to the renewal shift (see Subsection \ref{subsec:BD_renewal}) is completely irreducible. Clearly, diagram $B_{RS}$ is irreducible: for every level $n$ and every vertex $i \in V_n$ we can reach vertex $0 \in V_{n+i}$, and from vertex $0  \in V_{n+i}$ we can reach any other vertex $j \in V_{n + i + 1}$. This proof does not depend on the enumeration of the vertices. In other words, let $B'=(V',E')$ be isomorphic to $B_{RS}$ and $g_n \colon V_n \rightarrow V_n'$, $n \geq 0$ be the corresponding maps from the definition of isomorphism. Then $B'$ has the unique vertex $g_n(0)$ on each level $n$ such that $E(g_n(0),v) \neq \emptyset$ for all $v \in V'_{n+1}$. In the Bratteli diagram $B_{RS}$, for all $n$, the vertex $0 \in V_{n+i}$ is joined by a path of length $i$ to the vertex $i \in V_n$. Hence, for all $n, i \in \N_0$ the vertex $g_{n+i}(0) \in V'_{n+i}$ is joined by a path of length $i$ to the vertex $g_n(i) \in V_n'$. 
Moreover, for every $n$ and every vertex  $i' \in V'_n$, there exists $i \in V_n$ such that $i' = g_n(i)$. Thus, for all $n$ and all $i' \in V'_n$, we have $E(i', g_{n+i}(0)) = E(g_n(i), g_{n+i}(0)) \neq \emptyset$. Therefore, for every $n$ and every vertex $i'\in V'_n$ we can reach the corresponding vertex $g_{n+i}(0)$, and from $g_{n+i}(0)$ we can reach any vertex $j$ on level $V'_{n+i+1}$. Hence, the diagram $B'$ is irreducible.
\end{example}

The following proposition is a straightforward generalization of Example~\ref{ex:GBDclass1}. 

\begin{prop}\label{Prop:suff_cond_irred_isom}
    Let $B$ be a generalized Bratteli diagram. Assume that
    
    (i) for every $n \geq 0$ there is a vertex $w_n \in V_n$ such that $E(w_n,v) \neq \emptyset$ for all $v \in V_{n+1}$,

    (ii) for every vertex $w$ on every level $n$ there is a finite path from from $w$ to $w_m$ for some $m > n$.

    Then $B$ is completely irreducible.
\end{prop}

Note that condition (i) in Proposition~\ref{Prop:suff_cond_irred_isom} is not sufficient for a diagram $B$ to be irreducible. Indeed in the diagram $B_{\infty}$ (see Subsection \ref{subsec:diagramB-infty}), the first vertex of each level has an edge to all vertices on a level below, but the diagram is reducible.

\begin{example}[Relatively irreducible Bratteli diagrams] Bratteli diagrams of bounded size (see Subsection \ref{subsec:bdd}) are relatively irreducible. 
Indeed, let $B=(V,E)$ be a generalized Bratteli diagram of bounded size  with parameters $(t_n, L_n)$. For every $w \in V_0$, let $Y_w^+$ denote the set of all infinite paths which start at $w$ and then pass through the rightmost possible vertex on each level, i.e. for every $m \in \mathbb{N}$, the paths from $Y_w^+$ go through the vertex $w + \sum_{i = 0}^{m-1} t_i$ on level $m$.


\begin{prop}
    Let $B$ be a generalized Bratteli diagram of bounded size. Then $B$ is relatively irreducible. 
\end{prop}

\begin{proof}
    Let $(t_n,L_n)$ be the parameters for $B$ from the definition of a diagram of bounded size. For every $n \geq 1$ and $v \in V_n$, let $g_n(v) = v - \sum_{i = 0}^{n-1} t_{i}$ and $h_n$ be the corresponding bijection between the edges. Then in the new diagram $B'$, the images of the slanted paths from $Y_w^+ \subset X_B$ will be vertical paths, and any vertex $w \in V_0$ will be joined by finite paths only with vertices $v \leq w$ on lower levels.
\end{proof}

Clearly, all reducible generalized Bratteli diagrams such as diagrams of infinite odometers (see Subsection \ref{subsec:DIO}) or the diagram $B_{\infty}$ (see Subsection \ref{subsec:diagramB-infty}) are relatively irreducible.

\end{example}



Throughout the paper, by a neighbourhood of a point $x$ we mean any set $U$ that contains an open set containing $x$.

\begin{thm}\label{Thm:isomredus}
    Let $B$ be a generalized Bratteli diagram. 
    If $B$ has a point with a compact neighbourhood then $B$ is relatively irreducible. 
    
\end{thm}

\begin{proof} 
Enumerate the vertices $V_n$ of each level of $B$ by integers. Assume that $X_B$ possesses a point $x = (x_n)$ with a compact neighbourhood $U$. Since all cylinder sets form the base of topology and are closed, there exists a cylinder set $[(x_0, \ldots, x_n)]$ which contains $x$ and is compact. Hence, the set $[(x_0, \ldots, x_n)]$ can be represented as a standard Bratteli diagram, i.e. from the vertex $r(x_n) \in V_{n+1}$ one can reach only finitely many vertices 
    from $V_{i}$
    for each $i > n + 1$.
    Thus, the paths that pass through $r(x_n)$ form a cone similarly to the cones for the diagrams of bounded size. Let $v_{n + 2} \in V_{n+2}$ be the leftmost vertex that can be reached from $r(x_n)$. Let $v_{n+3}$ be the leftmost vertex that can be reached from $v_{n+2}$ and so on. Hence, we obtain a sequence of vertices $\{v_{i}\}_{i = n+2}^{\infty}$, through which passes the leftmost path of the downwards directed cone starting at $r(x_n)$. Consider a shift $\tau_i \colon V_i \rightarrow V_i$ for each level $i > n + 1$ which shifts vertex $v_i$ so that its number becomes equal to the number of $r(x_n)$. Hence, the leftmost path becomes vertical and we cannot reach from $r(x_n) \in V_{n+1}$ any vertex with number less that $r(x_n)$ on any level below. Thus, diagram $B$ is isomorphic to a reducible diagram and cannot be completely irreducible.
\end{proof}

\begin{example}
    Every diagram of bounded size (see Subsection \ref{subsec:bdd}) has a locally compact path space. The diagram $B_{RS}$ which corresponds to the renewal shift (see Subsection \ref{subsec:BD_renewal}) is completely irreducible and does not have points with a compact neighbourhood.
    Indeed, every path $x \in X_{B_{RS}}$ passes through the first vertex infinitely many times. Since the first vertex of each level of $B_{RS}$ has infinitely many outgoing edges, every cylinder set which contains $x$ is not compact.
    
    \end{example}

\begin{remark}
Theorem~\ref{Thm:isomredus} provides a 
necessary condition for a Bratteli diagram to be completely irreducible: such diagram cannot have paths with compact neighbourhouds. However, this condition is not sufficient. For the diagram $B_{\infty}$ (see Subsection \ref{subsec:diagramB-infty}), every vertex has infinitely many outgoing edges, hence there are no points with compact neighbourhoods, but the diagram is reducible. 
\end{remark}

For standard Bratteli diagrams, there is a notion of a simple diagram: it's a diagram such that for every level $n$ there is a level $m > n$ such that $E(v,w) \neq \emptyset$ for all $v \in V_m$ and $w \in V_n$. 
It is well known that a standard Bratteli diagram has a minimal tail equivalence relation if and only if the diagram is simple 
(see, e.g., \cite[Theorem 2.5]{Putnam2010}). For generalized Bratteli diagrams, one cannot use straightforwardly the notion of minimality. Since every set $V_n$ is countably infinite, and every vertex of a generalized Bratteli diagram can have only finitely many incoming edges, it is not possible for a vertex to be connected to all vertices on some level above. The following theorem motivates viewing completely irreducible generalized Bratteli diagrams as analogues of simple standard Bratteli diagrams.


\begin{theorem}\label{thm:Class1_minimalR}
    Let $B$ be a generalized Bratteli diagram and $\mathcal{R}$ be the tail equivalence relation on $B$. If $B$ is completely irreducible then $\mathcal{R}$ is minimal.
\end{theorem}  

\begin{proof} Assume the contrary. Suppose $\mathcal{R}$ is not minimal. Then there exists a point $x = (x_n) \in X_B$ such that $X_B \setminus \ov{\mathcal{O}(x)} \neq \emptyset$. Hence there exists a cylinder set $[(e_0, \ldots, e_n)]$ which belongs to $X_B \setminus \ov{\mathcal{O}(x)}$. Since $[(e_0, \ldots, e_n)] \cap {\mathcal{O}(x)} = \emptyset$, we have $r(e_n) \neq r(x_n)$.
Consider a diagram $B' = (V', E')$ which is isomorphic to $B$ with the corresponding maps $g_m \colon V_m \rightarrow V_m'$ such that $g_n(r(e_n)) = 0 \in V'_{n+1}$ and $g_m(r(x_m)) = 1 \in V'_{m+1}$ for all $m \geq n$. 
Then there is no finite path between $0 \in V'_{n+1}$ and vertex $1 \in V'_{m+1}$ for all $m > n$. Indeed, if such a path existed, then there would exist a finite path between $r(e_n)$ and $r(x_m)$ in $B$, hence the orbit of $x$ would visit $[(e_0, \ldots, e_n)]$ and we would get a contradiction. Thus, $B'$ is reducible and $B$ cannot be completely irreducible.
\end{proof}

\textit{Open problem.} Assume that the tail equivalence relation $\mathcal{R}$ on a generalized diagram $B$ is minimal. Is it true that $B$ is completely irreducible?

\medskip
Proposition \ref{prop:minloccomp} might be a step towards solving the open problem above.

\begin{prop}\label{prop:minloccomp}
If tail equivalence relation is minimal then either $X_B$ is locally compact or there is no path in $X_B$ with a compact neighbourhood.    
\end{prop}

\begin{proof}
    Assume that there is a path $x = (x_n) \in X_B$ without a compact neighbourhood. Then for all $n$, the vertex $s(x_{n})$ is joined by a path to a vertex on some level below which has infinitely many outgoing edges. 
    Since $\mathcal{R}$ is minimal, the orbit $\mathcal{O}(x)$ is dense in $X_B$. Assume that there is a point $y = (y_n) \in X_B$ with a compact neighbourhood. Then there exists $N \in \mathbb{N}$ such that the cylinder set $[(y_0, \ldots, y_N)]$ is compact. Since $\ov{\mathcal{O}(x)} = X_B$, there is a finite path between $r(y_N)$ and $s(x_m)$ for some $m \geq N$. Hence, there is a path between $r(y_N)$ and a vertex which has infinitely many outgoing edges. Hence, $[(y_0, \ldots, y_N)]$ cannot be compact. 
\end{proof}

\subsection{Connectedness of Bratteli diagrams}
In this subsection, we discuss connectedness of Bratteli diagrams as undirected graphs. Recall that a graph is called \textit{connected} if there exists a path between every pair of vertices. We discuss the relation between connectedness of a generalized Bratteli diagram, ireducibility, and topological properties of the tail equivalence relation.



\begin{remark}
    Note that for standard Bratteli diagrams, it is usually considered that the level $V_0$ consists of a single vertex $v_0$ which is connected to all vertices on level $V_1$. Thus, every vertex $v$ on any level $V_n$, for $n > 0$, is connected by a finite path to $v_0$. With such definition, standard Bratteli diagrams
    are always connected.
\end{remark}

The example below shows that, in general, the notions of irreducibility and connectedness for a generalized Bratteli diagram are independent. 

\begin{example} Diagram $B_{RS}$ (see Subsection \ref{subsec:BD_renewal}) is irreducible and connected. Diagram from Example \ref{ex:bdd_irred_nontrans} is irreducible but not connected. Diagram $B_{\infty}$ (see Subsection \ref{subsec:diagramB-infty}) is connected but not irreducible. Diagram from Example \ref{ex:bdd_red_nontrans} is reducible and not connected.
\end{example}


Proposition \ref{prop:Rconnect} demostrates the relation between connectedness of a generalized Bratteli diagram and topological transitivity of the corresponding tail equivalence relation.

\begin{prop}\label{prop:Rconnect}
    Let $B$ be a generalized Bratteli diagram and $\mathcal{R}$ be the tail equivalence relation on $B$. If $\mathcal{R}$ is topologically transitive then $B$ is connected.
\end{prop}

\begin{proof}
    Let $m,n \in \mathbb{N}_0$, $i \in V_m$ and $j \in V_n$. We show that there exists a finite path in the unordered graph $B$ which joins $i$ and $j$. Let $x = (x_k) \in X_B$ be a path such that $\ov{\mathcal{R}(x)} = X_B$. Then there exist paths $y  = (y_k), z  = (z_k) \in X_B$ such that $y, z \in \mathcal{R}(x)$ and path $y$ visits a cylinder that ends in vertex $i$, path $z$ visits a cylinder that ends in vertex $j$. In other words, $s(y_m) = i$ and $s(z_n) = j$. Since $y, z \in \mathcal{R}(x)$, there exist levels $M, N \in \mathbb{N}$ such that $y_k = x_k$ for all $k \geq M$ and $z_k = x_k$ for all $k \geq N$. Then a portion of path $y$ provides a finite path between vertex $i = s(y_m)$ and vertex $s(y_M) = s(x_M)$, a portion of path $x$ connects $s(x_M)$ and $s(x_N)$ and a portion of path $z$ connects $s(x_N) = s(z_N)$ and $s(z_n) = j$.
\end{proof}

\begin{corol}
    Let $B$ be an irreducible stationary generalized Bratteli diagram. Then $B$ is connected.
\end{corol}

\begin{proof}
    The proof follows directly from Proposition \ref{prop:Rconnect} and Theorem \ref{ThmBJKStoptrans}.
\end{proof}

\begin{remark}
    Note that the converse to Proposition \ref{prop:Rconnect} is not true. Indeed, diagram form Example \ref{ex:connected_non-transR} is connected, but its tail equivalence relation is not topologically transitive.
\end{remark}


\textbf{Acknowledgements.} The author is very grateful to the colleagues and collaborators, 
especially, S. Bezuglyi, H. Bruin, P. Jorgensen, J. Kwiatkowski, P. Oprocha, S. Radinger, S.~Sanadhya, T. Raszeja  for valuable discussions.
This research was partially supported by a subsidy from the Polish Ministry of Science and Higher Education for the AGH University of Krakow.

\bibliographystyle{alpha}
\bibliography{ReferencesGBD2}

\end{document}